\documentclass[11pt]{article}
\usepackage{amsmath,amsthm,amscd,amssymb}
\usepackage{amsrefs}
\usepackage[margin=1.25in]{geometry}

\usepackage[colorlinks, citecolor = blue]{hyperref}
\usepackage[capitalize, nameinlink]{cleveref}

\newcommand{\Q}{\mathbb Q}
\newcommand{\Z}{\mathbb Z}

\newcommand{\C}{\mathbb C}
\newcommand{\GL}{\mathrm{GL}}

\newcommand{\A}{\mathbb A}
\newcommand{\N}{\mathbb N}
\newcommand{\Cl}{\mathrm{Cl}}

\newcommand{\calO}{\mathcal O}
\newcommand{\calI}{\mathcal I}

\newcommand{\calE}{\mathcal E}
\newcommand{\calM}{\mathcal M}
\newcommand{\calS}{\mathcal S}
\newcommand{\frakN}{\mathfrak N}
\newcommand{\frakp}{\mathfrak p}
\newcommand{\frakq}{\mathfrak q}
\newcommand{\frako}{\mathfrak o}
\newcommand{\num}{\mathrm{num}}
\newcommand{\bmx}{\begin{pmatrix}}
\newcommand{\emx}{\end{pmatrix}}
\renewcommand{\mod}{\, \, \mathrm{mod} \, \,}
\def\vol{{\rm vol}}
\def\alg{{\rm alg}}
\def\bs{{\backslash}}

\newtheorem{lem}{Lemma}
\numberwithin{lem}{section}
\newtheorem{prop}[lem]{Proposition}
\newtheorem{thm}[lem]{Theorem}
\newtheorem{rem}[lem]{Remark}
\newtheorem{ex}[lem]{Example}

\newtheorem*{thma}{Theorem A}
\newtheorem*{thmb}{Theorem B}
\newtheorem*{corc}{Corollary 1}
\newtheorem*{cord}{Corollary 2}
\newtheorem*{exe}{Example 3}

\numberwithin{equation}{section}

\begin{document}

\title{The Jacquet--Langlands correspondence, Eisenstein congruences,
and integral $L$-values in weight 2}

\author{Kimball Martin\footnote{Department of Mathematics, University of
Oklahoma, Norman, OK 73019}}

\date{\today\footnote{This version is a corrected and annotated 
version of the published version.  Annotations (including explanations of primary
corrections) are made in footnotes.
The main corrections are the additional hypotheses that $p \nmid h_F$ and $h_F$
is odd in \cref{thm1}.  We thank Jack Shotton for pointing out the error when
$p | h_F$.  While one can remove these hypotheses with some modifications
to the argument, in order to minimize changes to this document 
we defer such explanations to an upcoming joint paper with Satoshi
Wakatsuki.  The results in the introduction are unchanged.}}

\maketitle

\begin{abstract}
We use the Jacquet--Langlands correspondence to
generalize well-known congruence results of Mazur on Fourier coefficients and
$L$-values of elliptic modular forms for 
prime level in weight 2 both to nonsquare level and to Hilbert modular forms.
\end{abstract}

A celebrated result of Mazur says that, for $N$ a prime and $p$ a prime dividing
(the numerator of) $\frac{N-1}{12}$, there exists a cusp form $f \in S_2(N) = S_2(\Gamma_0(N))$ congruent
to the Eisenstein series $E_{2,N}$ of weight 2 and level $N$ mod $p$ \cite[II(5.12)]{mazur1}.  Further, if $p \ne 2$, 
one has a congruence
for the algebraic part of the  central $L$-value $L(1, f_{K})=L(1, f) L(1, f \otimes
\eta_K)$, where $\eta_K$ is the quadratic character associated to a quadratic field $K/\Q$ \cite{mazur2}.  For instance, if $N=11$ and $K$ is not split at 11, there is one
cusp form $f \in S_2(N)$ and one gets that $L^{\alg}(1, f_K) \not \equiv 0 \mod 5$
if and only if $5 \nmid h_K$.  (If $K$ splits at 11, the root number
is $-1$ so $L(1, f_K) = 0)$.
A form of Mazur's $L$-value result was reproved by Gross \cite{gross} for $K/\Q$ imaginary quadratic using quaternion algebras and the height pairing, whereas Mazur used modular symbols. 
Ramakrishnan pointed out to me
that one can also deduce this from his average $L$-value formula with 
Michel \cite{michel-ramakrishnan}. (Note Gross's argument also involves an
averaging type procedure, so these two arguments are not entirely different in spirit.)

In this article, we use the Jacquet--Langlands correspondence and an explicit 
$L$-value formula to extend these results of Mazur both to more general levels and to parallel weight 2 Hilbert modular forms over a totally real field $F$ with $K/F$ a quadratic 
CM extension.
For simplicity, we only state our results precisely for $F=\Q$ in this introduction.

We first discuss the Hecke eigenvalue congruence result and a nonvanishing
$L$-value result, and will state the more precise result
on  $L$-value congruences below in \hyperlink{thmb}{Theorem B}.

\begin{thma} \hypertarget{thma}
 Let $N$ be a nonsquare, and write $N=N_1N_2$ where $(N_1, N_2) = 1$
and $N_1$ has an odd number of prime factors, all of which occur to odd exponents.
Let $p$
be a prime dividing (the numerator of) $\frac 1{12} {\varphi(N_1)N_2} \prod_{q | N_2} (1+q^{-1})$ and $\frakp$ a prime of $\overline \Q$ above $p$.  
Then there exists $f \in S_2(N)$ which is an eigenform for all $T_n$ with $(n,N) = 1$ and 
 is congruent to the
Eisenstein series $E_{2,N} \mod \frakp$ away from $N$.

Suppose moreover $N=N_1$ with $N_1$ squarefree and that $p | \frac{\varphi(N_1)}{24}$.  
Then we can take $f$ to be a newform such that $f \equiv E_{2,N} \mod \frakp$
(as $q$-expansions).
If $K/\Q$ is an imaginary quadratic field not split at any prime dividing $N$ such that
$p \nmid h_K$, then there exists such an $f$ with $L(1,f_K) \ne 0$.
\end{thma}

Here
$E_{2,N}$ is the Eisenstein series of weight 2 and level $N$ defined by
$E_{2,N}(z) = \sum_{d | N} \mu(d) d E_2(dz)$, where $\mu$ is the M\"obius function
and $E_2$ is the weight 2 Eisenstein series for $\mathrm{SL}(2,\Z)$ 
normalized so that the Fourier coefficient of $q$ is 1.

The statement about congruence away from $N$ means that $f$ and $E_{2,N}$ have
the same  Hecke eigenvalues mod $\frakp$ for $T_\ell$ for any prime $\ell \nmid N$.  
If $N=N_1$ is squarefree, then $E_{2,N}$ has constant term $\frac{\varphi(N)}{24}$, and
the only reason we need to assume $p | \frac{\varphi(N)}{24}$ rather than $p | \frac{\varphi(N)}{12}$ in the second part is so the constant term of $E_{2,N}$ will be $0 \mod p$.  E.g., if
$N=73$ there is a rational  newform $f \in S_2(73)$ all of whose Hecke eigenvalues
match with those of $E_{2,73}$ mod $2$, but the constant term of $E_{2,73}$ is
$1 \mod 2$ so they are not congruent as $q$-expansions.

Note $E_{2,N} = E_{2,N'}$
where $N' = \prod_{p | N} p$ is the ``powerfree part'' of $N$ so $E_{2, N}$ really has level $N'$, and 
$E_{2,N}$ is an eigenfunction of the Hecke operators $T_\ell$ on $M_2(N)$ 
for all primes $\ell$ having eigenvalue $\ell + 1$ when $\ell \nmid N$ and eigenvalue 1 when $\ell | N$.
One might prefer to replace $E_{2,N}$ with a form whose $T_\ell$ eigenvalue is 0
when $\ell^2 | N$ and look for a congruence mod $\ell$ as well.  
However, since we do not know the exact level of $f$ in general,
this does not make a difference for the statement of the
theorem.  Still, we give non-squarefree examples below where one can replace 
$E_{2,N}$ with a different
Eisenstein series to get a congruence of all Fourier coefficients.

The Hecke eigenvalue congruence in the $N=N_1$ squarefree case of 
\hyperlink{thma}{Theorem A}
was also proved by Ribet (unpublished, announced in 2010) under the
additional assumptions $p \nmid N$ and $p \ge 5$.  In fact, Ribet also obtained a
converse (there are no other such $p \nmid N$ with $p \ge 5$) and more generally addressed finding 
newforms congruent to $E_{2,N}$ away from $N$ with specified $p_i$-th Fourier coefficient 
$a_{p_i} \in \{ \pm 1 \}$ for each 
$p_i | N$.  This was studied further by Yoo in \cite{yoo:1}, which also explains Ribet's
work in detail.
(When $N=N_1$ is squarefree, our \hyperlink{thma}{Theorem A} only concerns finding
newforms with $a_{p_i} = 1$ for each $p_i | N$.)  
Their arguments involve studying Eisenstein ideals using Jacobian calculations 
and then appealing to the Jacquet--Langlands correspondence.
See also \cite{yoo:2} for related results when $N$ is not squarefree.

In contrast, our proof follows from using a (generalized) Eichler mass formula to 
construct a quaternionic cusp form $\phi$ ``of level $N$'' 
on the definite quaternion algebra $B$ ramified at each prime dividing $N_1$ such
that $\phi \equiv 1 \mod \frakp$.  (Here $\phi$ is just
a function on a certain finite set of ideal classes with the cuspidality
criterion amounting to a linear relation among the values of $\phi$.)
In addition to our method 
being more straightforward, we
have no trouble dealing with $N$ not squarefree and arbitrary $p$, or extending
this to Hilbert modular forms as in \cref{thm1} (though in the Hilbert modular
case we make a
simplyfing assumption that our totally real field $F$ satisfies $h_F = h_F^+$---see
\cref{rem:thm1}).  We do not get a converse or consider $a_{p_i}=-1$, though
see \cref{rem:converse} regarding the converse.

Here are some simple consequences of \hyperlink{thma}{Theorem A}.

\begin{corc} If $p > 2$, then there exists $f \in S_2(p^3)$ such that $f$ and $E_{2,p^3} = E_{2,p}$
are congruent$\mod \frakp$ away from $p$, for $\frakp$ a prime above $p$ in $\overline \Q$.  
\end{corc}

For example, when $p=3$, there is a unique normalized (new)form $f \in S_2(27)$, which has Fourier expansion
\[ \sum a_n q^n = q - 2q^{4} - q^{7} + 5q^{13} + 4q^{16} - 7q^{19} - 5q^{25} + 2q^{28} - 4q^{31} + 11q^{37} + 8q^{43} + \cdots \]
and it satisfies $a_\ell \equiv \ell+1 \mod 3$ for all primes $\ell \ne 3$.  In fact
we have $f(z) \equiv E_{2,3}(z) - E_{2,3}(3z) \mod 3$. (Note the constant term of the right
hand side is $0$.)

We remark that the analogous statement of the corollary is not true in level $p^2$---e.g., 
for $N=N_1 = 25$,
then prime $p=5$ divides the numerator of $\frac 1{12} \phi(N_1)$, but $S_2(25) = \{ 0 \}$. 
 
In the case of $p=2$, the theorem says that there exists $f \in S_2(32)$ which is congruent to $E_{2,2}$ mod $2$ away from $2$.  Here $S_2(32)$
is 1 dimensional, and spanned by the form
\[ \sum a_n q^n = q - 2q^{5} - 3q^{9} + 6q^{13} + 2q^{17} - q^{25} - 10q^{29} - 2q^{37} + 10q^{41} + 6q^{45} - 7q^{49} + \cdots. \]
In fact, $f(z) \equiv E_{2,2}(z) - E_{2,2}(2z) \mod 2$.

\begin{cord} Let $p, q$ be distinct primes with $p > 3$.  Then there exists $f \in S_2(qp^2)$
such that $f$ and $E_{2,q}$ are congruent $\mathrm{mod}$ $\frakp$ away from $pq$, for 
$\frakp$ a prime of $\overline \Q$ above $p$.
\end{cord}

For instance, $S_2(50) = S_2(50)^\mathrm{new}$ has dimension 2, and one eigenform 
$f \in S_2(50)$ has Fourier expansion
\[ q + q^{2} - q^{3} + q^{4} - q^{6} - 2q^{7} + q^{8} - 2q^{9} - 3q^{11} - q^{12} + 4q^{13} - 2q^{14} + q^{16} + 3q^{17} + \cdots \]
In this case it happens that $f$ and $E_{2,50}$ are congruent mod $5$ everywhere
away from 5, and $f(z) \equiv E_{2,10}(z) - E_{2,10}(5z) \mod 5$.

\medskip
The $L$-value result at the end of \hyperlink{thma}{Theorem A} is a direct 
generalization of \cite[Cor 11.8]{gross}, and follows from our proof of
\cref{thm1} and Waldspurger's formula \cite{wald}---see \cref{prop:nonv}.
The $p=2$ case
has consequences in the spirit of Goldfeld's conjecture:

\begin{exe} Let $C$ be the unique-up-to-isogeny elliptic curve over $\Q$ of conductor $N=17$.  
Then among negative prime discriminants $-d$, the 
central value $L(1, C_{-d})$ for the quadratic twist $C_{-d}$ is 
nonzero exactly $50\%$ of the time.  
\end{exe}

In fact, we can say precisely when $L(1,C_{-d}) \ne 0$ for $-d < 0$ a prime discriminant. 
We know $C$
is associated to the unique normalized form $f \in S_2(17)$.  We take $p=2$ 
in our theorem and conclude $L(1,f_K) \ne 0$ when $17$ is inert in 
$K=\Q(\sqrt{-d})$ and $2 \nmid h_K$.  However, for $d$ prime, Gauss' genus
theory implies $2 \nmid h_K$.  Thus, if $17$ is inert in $K$, 
$L(1,f_K) = L(1,C)L(1,C_{-d}) \ne 0$.
On the other hand, if $17$ is split in $\Q(\sqrt{-d})$, then $L(1,C_{-d})=0$ because the root 
number is $-1$.  In particular, at least 50\% of these prime quadratic twists $C_{-d}$ have 
finitely many rational points.  

\medskip
Now we discuss congruences of $L$-values.  Let $K/\Q$ be a quadratic extension and
$\chi$ an ideal class character of $K$.  Let $L(s,f,\chi)$ denote the twist $L(s, f_K \otimes \chi)$
of the base change $f_K$ of $f$ to $K$ (a degree 2 $L$-function over $K$), or equivalently
$L(s, f \times \theta_\chi)$ the degree 4 Rankin--Selberg $L$-function over $\Q$ of $f$ times
the $\GL(2)$ theta series $\theta_\chi$ attached to $\chi$.
  When $\chi=1$ is trivial, this $L$-function is just
$L(s,f,1) = L(s,f_K) = L(s, f) L(s, f \otimes \eta_K)$.

Given a congruence of modular forms, say as in \hyperlink{thma}{Theorem A}, we expect a congruence of algebraic parts of the special $L$-values $L(1, f, \chi)$.
In the case where one of the modular forms
is an Eisenstein series, this was studied by Vatsal and Heumann \cite{vatsal},
\cite{hv} (after Mazur \cite{mazur2} and Stevens \cite{stevens}) using modular symbols.

The Jacquet--Langlands correspondence, together with an explicit
$L$-value formula of earlier joint work with Whitehouse \cite{mw}, 
also yields a congruence of central $L$-values when one starts with
a congruence of \emph{quaternionic} modular forms.  We just consider the case
where one of these quaternionic modular forms is constant (i.e., corresponds to
$E_{2,N}$) but the same approach can also be used when one starts with two quaternionic
cusp forms which are congruent mod $\frakp$. 

\begin{thmb} \hypertarget{thmb}
Suppose $N$ is squarefree product of an odd number of primes 
and $f \in S_2(N)$ corresponds via Jacquet--Langlands to a 
quaternionic cusp form $\phi$ of level $N$ such that
$\phi \equiv 1 \mod \frakp^r$ for a prime $\frakp$ of $\overline \Q$ above $p$.
Then we may normalize $L^\alg(1, f, \chi)$ so that (i) it is an algebraic integer,
and (ii) there exist an algebraic integer $c_\phi$ invertible mod $\frakp$ such that,
for any imaginary quadratic field $K$ inert at each prime dividing $N$ and
any ideal class character $\chi$ of $K$, we have
\[ L^\alg(1, f, \chi) \equiv 
\begin{cases} 
|c_\phi|^2 h_K^2 \mod \frakp^r & \chi = 1,\\
0 \mod \frakp & \chi \ne 1.
\end{cases} \]
\end{thmb}

Given $f$, the normalization of $L^\alg(1,f,\chi)$ depends on $K$ and $\phi$ 
in a simple way, and does not depend on $\chi$---see \eqref{eq:Lalg-def}.  The
integer $c_\phi$ can also be determined in a simple way from $\phi$.

Note this $L$-value congruence can be viewed as a congruence between $L^\alg(1,f,\chi)$
and $L^\alg(1,E_{2,N},\chi)$, though we have not written it in this form. 
For instance, if $\chi$ is trivial, then
$L(1,E_{2,N},\chi)$ is essentially $L(0,\eta_K) L(1, \eta_K)$, which is essentially $h_K^2$.

Vatsal and Heumann \cite{vatsal}, \cite{hv} show that if $f$ is congruent to $E_{2,N}$
mod $\frakp^r$, $p > 2$ and $p \nmid N$, 
then there is a congruence of special values between algebraic parts of $L(1,f,\chi)$ and
$L(1,E_{2,N},\chi)$.
In fact they treat higher weights, other Eisenstein series and non-central special values (when $k \ge 4$).  However their normalization of algebraic parts
of special values makes use of canonical periods, which are only defined up
to $\frakp$-adic units, whereas our normalization determines $L^\alg(1,f,\chi)$ 
uniquely.  On the other hand,
 we do not require any restriction on $p$, and we can easily treat
 Hilbert modular forms also.
 
Our method is similar in spirit to Gross's approach \cite{gross},
and also an approach by Quattrini \cite{quattrini} using half-integral
weight modular forms.  (These works also make use of definite quaternion algebras.)
In fact we get a period congruence, which is stronger than the $L$-value congruence as
the period also accounts for the sign of the ``square root'' of $L^\alg(1,f_K)$.

The main deficiency in our result is that we require a congruence of \emph{quaternionic} modular forms, which is a priori stronger than a congruence of
elliptic (or Hilbert) modular forms.  However, we expect that these two notions of
congruence are equivalent (see 
\cref{rem:converse} and \eqref{eq:hyp}).  This issue is also present in Quattrini's approach,
who used work of Emerton \cite{emerton}
to show these notions are equivalent when $r=1$ and $N$ is prime 
\cite[Thm 3.6]{quattrini}.   Hence our  \hyperlink{thmb}{Theorem B} covers the case addressed
in \cite{mazur2}.
We can also show that, for 
$r=1$, under the hypotheses of \hyperlink{thma}{Theorem A}, if there is a unique cuspidal
newform $f \equiv E_{2,N} \mod \frakp$, then it corresponds to a quaternionic
form $\phi \equiv 1 \mod \frakp$.    The difficulty in general is separating quaternionic 
eigenforms with congruent eigenvalues.
 
\medskip
Now we outline the contents and briefly remark on other related literature.

In \cref{sec1}, we explain some preliminaries on weight 2 quaternionic modular forms for
totally definite quaternion algebras $B$ over totally real number fields $F$.  These
will be functions on the finite set of ideal classes $\Cl(\calO)$, for some order $\calO$
of $B$.

In \cref{sec2}, we use the Eichler mass formula to show the existence of a
quaternionic cusp form $\phi$ congruent to the constant function 1 modulo suitable primes.
Then we apply the Jacquet--Langlands correspondence to get a Hilbert cusp form 
$f$ corresponding to $\phi$ whose Hecke eigenvalues are congruent to those
of a Hilbert Eisenstein series associated
to the constant function on $\Cl(\calO)$.  This (\cref{thm1}) is the first main result, which specializes to the Hecke eigenvalue congruence statements 
in  \hyperlink{thma}{Theorem A}.

We note recent work
of Berger, Klosin and Kramer \cite{bkk} gives an algebraic approach to  
counting congruences of Hecke eigenvalues in more general settings, which yields
a refinement of Mazur's congruence of Fourier coefficients result
for $S_2(N)$ with $N$ is prime.  
In the case of Hilbert modular forms, some results are already known about
 finding primes of congruence between two \emph{cusp} forms, 
e.g., the work of Ghate \cite{ghate} generalizing a result of
Hida \cite{hida} for elliptic cusp forms, but we are not aware of results along the lines 
of \cref{thm1} guaranteeing the existence of Hilbert cusp forms congruent to 
Eisenstein series.

Then in \cref{sec3}, we use Waldspurger's formula \cite{wald} relating central $L$-values to periods on quaternion algebras.  This immediately gives the nonvanishing
$L$-value statement in  \hyperlink{thma}{Theorem A}.  To get the precise congruence
in  \hyperlink{thmb}{Theorem B}, we use a more explicit version of Waldspurger's
formula from \cite{mw}.
That $L$-value formula applies to
arbitrary quadratic extensions $K/F$ ($F$ totally real or not), but we need to restrict to $K/F$ CM here in order for the relevant period to lie on the definite quaternion algebra $B$.
Again, we are not aware of such $L$-value congruences in the Hilbert modular case.

On the other hand, there are some existence results of a different nature about even weight cusp 
forms  with $L^\alg(1, f, \chi) \not \equiv 0 \mod \frakp$ using average
value formulas, e.g., the aforementioned work of Michel and Ramakrishnan
\cite{michel-ramakrishnan} for elliptic cusp forms and joint work of File and Pitale with the present 
author \cite{fmp} for Hilbert cusp forms (though excluding parallel weight 2 for
simplicity).  These results are of just the form that
for a given $\frakp$ and any suitably large level  $\frakN$, there exists some
$f \in S_k(\frakN)$ such that $L^\alg(1, f, \chi) \not \equiv 0 \mod \frakp$, but 
the bound on the level \emph{depends on $K$}.  
I am not aware of any other results on the vanishing of $L^\alg(1, f, \chi)$ mod $\frakp$ for $\chi$ nontrivial when the sign of the functional equation is $+1$.

\medskip
This project grew out of several discussions with Dinakar Ramakrishnan about nonvanishing 
$L$-values, and I am
grateful to him for leading me to think about these things.  This work
was done in part while visiting Osaka City University with the support of a JSPS Invitation 
Fellowship (Long Term, L14518), and in part with support from a Simons Collaboration 
Grant.  I also thank Nicolas Billerey, Masataka Chida, Lassina Demb\'el\'e, Catherine Hsu, Ariel Pacetti and
Hwajong Yoo for helpful comments and directions to related literature.  I am especially grateful
to the referee for a prompt yet careful reading and finding a gap in an earlier version.

%

\section{Quaternionic modular forms} \label{sec1}

%


Throughout, we fix the following notation.  Let $F$ be a totally real number field of degree
$d$, and $\frako_F$ its ring of integers.
Let $\frakN=\frakN_1 \frakN_2$ be a nonzero integral ideal in $\frako_F$
such that $\frakN_1=\frakp_1^{e_1} \cdots \frakp_{r}^{e_{r}}$ and  
$\frakN_2 = \frakq_1^{f_1} \cdots \frakq_s^{f_s}$ where the $\frakp_i$'s and $\frakq_j$'s are 
distinct prime ideals of $\frako_F$, $e_i, f_j \in \N$, each $e_i$ is odd, and $r \equiv d \mod 2$.
Note that $\frakN$ can be any nonzero ideal when $d$ is even, and $\frakN$ can be any ideal
which is not a square when $d$ is odd.  The letter $\ell$ will denote a finite prime of $F$.

Denote by $B$ the unique (up to isomorphism) totally definite quaternion algebra over $F$ 
ramified at each $\frakp_i$ and no other finite prime.  
Let $\calO$ be an order of level $\frakN$ in $B$ such that $\calO_{\frakq_i}$
is conjugate to $R_0(\frakq_i^{f_i})$, the subring of $M_2(\frako_{F,\frakq_i})$ with lower
left entry $0$ mod $\frakq_i^{f_i}$, for $1 \le i \le s$.  Note this is an Eichler order if $\frakN_1$ is squarefree---otherwise
it is the intersection of an Eichler order of level $\prod \frakp_i \cdot \frakN_2$ with
 an order, which is unique up to isomorphism, of level $\frakN_1$.

We want to work with a certain space of automorphic forms on $B^\times$.  Adelically,
we will be looking at functions on
\begin{equation}
 \Cl(\calO) = B^\times \bs \hat B^\times / \hat \calO^\times = B^\times  \bs B^\times(\A_F) /
(\hat \calO^\times \times B^\times_\infty).
\end{equation}
Here $\Cl(\calO)$ can naturally be viewed as the (finite) set of invertible (locally principal)
 right ideal classes of $\calO$.  We denote the class number $\# \Cl(\calO)$ by $h(\calO)$.
Let  $\calI_1, \cdots, \calI_h$ be a set of representatives for the right ideal
classes of $\calO$, and let $x_1, \ldots, x_h$ be a corresponding set of representatives
for the double coset classes $B^\times \bs \hat{B}^\times / \hat \calO^\times$.

We consider the space of (weight 2, or weight 0, depending on convention)
quaternionic modular forms on $\calO$ defined by
\[ M(\calO) = \{ \phi : \Cl(\calO) \to \C \}, \]
a complex vector space of dimension $h(\calO)$.
To work adelically, we will typically view $\phi \in M(\calO)$ as a
function on $\hat B^\times$ which is left invariant by $B^\times$
and right invariant by $\hat \calO^\times$.
For simplicity, we will work with the subspace of forms with trivial central 
character:\footnote{In the published version, we incorrectly wrote
$\calM$ as the space of forms which are constant on the center,
rather than invariant.   That space is rather the span of $M(\calO, 1)$ 
together with forms in $M(\calO)$ which vanish on $x=1$.  This correction
is what forces the additional hypothesis of $p \nmid h_F$ in \cref{thm1}.}
\begin{equation} \label{eq:M-def}
 \calM =  M(\calO, 1) = \{ \phi \in M(\calO) : \phi(z x) = \phi(x) \text{ for } z \in \hat F^\times,
 \, x \in \Cl(\calO) \}.
\end{equation}
Note that $M(\calO, 1) = M(\calO)$ if $F$ has class number 1, because in this case
$\hat F^\times = F^\times \hat \frako_F^\times \subset B^\times \hat \calO^\times$.

We make $\calM$ into an inner product space
by defining
\[ (\phi, \phi') = \int_{\A_F^\times B^\times \bs B^\times(\A_F)} \phi(g) \overline{\phi'(g)} \, dg = (4\pi^2)^d \int_{\hat F^\times  B^\times \bs \hat B^\times} \phi(g) \overline{\phi'(g)} \, dg, \]
where  $dg$ denotes the Haar measure on the relevant quotient 
induced by the product of local
Tamagawa measures on $B^\times_v$ and $F^\times_v$ and the counting
measure on $B^\times$.  The inner product converges by compactness.  

We can write $\hat F^\times B^\times x_i \hat \calO^\times = \bigsqcup_z \hat B^\times z x_i \hat
\calO^\times$, where $z \in \hat F^\times/(\hat F^\times \cap B^\times x_i \hat \calO^\times x_i^{-1} )$.  The latter set is finite of size at most $h_F$ since
$F^\times \hat \frako_F^\times \subset \hat F^\times \cap B^\times x_i \hat \calO^\times x_i^{-1}$.    Denote its cardinality by $c_i$.

Since $\phi$ and
$\phi'$ are right invariant by 
$\hat \calO^\times$, we may write
\[ ( \phi, \phi' ) = \sum_i \omega_i^{-1} \phi(\calI_i) \overline{\phi'(\calI_i)}, \qquad \phi, \phi' \in M(\calO), \]
where $\phi(\calI_i) = \phi(x_i)$ and
\[  \omega_i = (4\pi^2)^{-d}
 \vol(\hat F^\times B^\times  \bs \hat F^\times B^\times x_i \hat \calO^\times)^{-1} c_i. \]
Note
\[  \vol(\hat F^\times B^\times  \bs \hat F^\times B^\times x_i \hat \calO^\times)
=  \vol(\hat F^\times B^\times  \bs \hat F^\times B^\times x_i \hat \calO^\times x_i^{-1})
= \vol(\hat \frako_F^\times B^\times \bs B^\times  x_i \hat \calO^\times x_i^{-1}) \frac{c_i}{h_F}. \]
Since $\calO_\ell(\calI_i) = x_i \hat \calO^\times  x_i^{-1} \cap B^\times$, we have
\begin{equation} 
\omega_i = \frac{[\calO_\ell(\calI_i)^\times:\frako_F^\times]}
{\vol(\hat \calO^\times/\hat \frako_F^\times)} \frac{h_F}{(4\pi^2)^d}.
\end{equation}

One can obtain the following Tamagawa volume computations from \cite{vigneras}.  

First, $\vol(\frako_{F_v}^\times) = \Delta_{F_v}^{-\frac 12}$ for $v < \infty$.


For finite $v \nmid \frakN$, we have
\[ \vol(\calO^\times_v / \frako_{F_v}^\times) = 
(1-q_v^{-2}) \vol(\frako_{F_v}^\times)^3 =
L(2, 1_{F_v})^{-1} \vol(\frako_{F_v}^\times)^3. \]

For $v = \frakp_i | \frakN_1$, we have
\[  \vol(\calO^\times_v / \frako_{F_v}^\times) = (1-q_v^{-1})^{-1} q_v^{-e_i} L(2, 1_{F_v})^{-1}\vol(\frako_{F_v}^\times)^3. \]

For $v = \frakq_j | \frakN_2$, we have that 
\[ [ \GL_2(\frako_{F,\frakq_j}) : R_0(\frakq_j^{f_j})^\times ] = q_v^{f_j}(1+q_v^{-1}), \]
and therefore
\[  \vol(\calO^\times_v / \frako_{F_v}^\times) =  q_v^{-f_j}(1+q_v^{-1})^{-1}L(2, 1_{F_v})^{-1} \vol(\frako_{F_v}^\times)^3. \]

Putting together the local measures gives 
\begin{equation}
 \vol(\hat \calO^\times/\hat \frako_F^\times) = \frac{\Delta_F^{-3/2}}{N(\frakN) \zeta_F(2)} 
\prod_{v | \frakN_1} \frac 1{1- q_v^{-1}} \prod_{v| \frakN_2} \frac 1{1+q_v^{-1}} .
\end{equation}

Note that computing $(\phi, \phi)$ for the constant function $\phi_0 = 1$ gives
\[ \sum \omega_i^{-1} = (\phi_0, \phi_0) = (2\pi)^{2d} \vol(\hat F^\times B^\times \bs \hat B^\times) = \vol(B^\times \A_F^\times \bs B^\times(\A_F)) = 2.  \]

If we put $w_i = [\calO_\ell(\calI_i)^\times :\frako_F^\times] = 
\frac{(2\pi)^{2d} \vol(\hat \calO^\times/\hat \frako_F^\times)}{h_F} \omega_i$ 
and define the 
normalized paring 
\begin{equation} \label{eq:normpair}
 [\phi, \phi'] = \frac{h_F}{(2\pi)^{2d} \vol(\hat \calO^\times/\hat \frako_F^\times)}
 (\phi, \phi') = \sum w_i^{-1} 
\phi(\calI_i) \overline{\phi'(\calI_i)}, \qquad \phi, \phi' \in M(\calO),
\end{equation}
 then we recover a generalized form of the usual Eichler mass formula,
\begin{align} \nonumber
m(\calO) = \sum w_i^{-1}  = [\phi_0, \phi_0] &= 2^{1-2d} \pi^{-2d} |\Delta_F|^{3/2} h_F
 \zeta_F(2)  N(\frakN) 
\prod_{v | \frakN_1} (1 - q_v^{-1}) \prod_{v| \frakN_2} ({1+q_v^{-1}})  \\
&= 2^{1-d} h_F |\zeta_F(-1)|  N(\frakN) 
\prod_{v | \frakN_1} (1-q_v^{-1}) \prod_{v| \frakN_2} ({1+q_v^{-1}}).
\label{eq:mass}
\end{align}
Here $N(\frakN)$ is the level (reduced norm) of $\frakN$.
The rational number $m(\calO)$ is called the mass of $\calO$.
When $F=\Q$, $\frakN = (N_1)$ and $\frakN_2 = (N_2)$, this simplifies to
\begin{equation}
m(\calO) = \sum w_i^{-1} =   \frac{\varphi(N_1)}{12}  N_2 \prod_{p | N_2}(1+p^{-1}).
\label{eq:massQ}
\end{equation}

Now we want to define the Eisenstein and cuspidal subspaces of $\calM$.
The Eisenstein space will be generated by the one-dimensional representations of
$\hat B^\times$, which all factor through the reduced norm map $N: \hat B^\times \to
\hat F^+$.  The reduced norm induces a surjective map $N: \Cl(\calO) \to \Cl^+(\frako_F)$
to the narrow class group of $F$.
Define the Eisenstein subspace $\calE$ of $\calM$ to be the subspace of all 
$\phi \in \calM$ which factor through $N$, and the cuspidal space $\calS$ to be the
orthogonal complement of $\calE$ in $\calM$.  Note for a character $\psi$ of 
$\Cl^+(\frako_F)$, we have  $\psi \circ N \in \calM$ if and only if $\psi^2 = 1$. 
Hence we can describe $\calS$ in terms of our normalized pairing \eqref{eq:normpair} 
by\footnote{This description of $\calS$ was corrected.}
\[ \calS = \{ \phi \in \calM : [\phi, \psi \circ N] = 0
\text{ for all characters } \psi: \Cl^+(\frako_F) \to \C \text{ s.t. } \psi^2 = 1 \}. \]

One has in the usual way Hecke operators $T_\ell$ for each prime $\ell \nmid \frakN$, 
which commute with each other and are self-adjoint with respect to the inner product.  
One can also describe the action on $\calM$ (or $M(\calO)$) in terms of 
(generalized) Brandt matrices.
Hence $\calM$,
and also $\calS$, has a basis consisting of eigenforms for each such $T_\ell$.  
Via the Jacquet--Langlands correspondence, each eigenform $\phi \in \calS$ transfers to an
eigenform $f \in S_2(\frakN)$, the space of parallel weight 2 Hilbert modular forms
of level $\frakN$, with the same Hecke eigenvalues away from $\frakN$.

Suppose now $\ell | \frakN_1$ but $\ell^2 \nmid \frakN_1$.  Then
one can define a Hecke operator by $(T_\ell \phi)(x) = \phi(x \varpi_{B_\ell})$ where
$\varpi_{B_\ell}$ is a uniformizer for  $B_\ell$.  This corresponds to the double coset
$\calO_\ell \varpi_{B_\ell} \calO_\ell = \varpi_{B_\ell} \calO_\ell$.  
Now one has bases for $\calM$ and $\calS$ which are eigenforms
 for all $T_\ell$ with $\ell \nmid \frakN_2$ and $\ell^2 \nmid \frakN_1$, and the Jacquet--Langlands lift of such an eigenform $\phi$
 will be a modular form $f \in S_2(\frakN)$ which, at all such $\ell$, is new and is
 an eigenform for $T_\ell$ with the same eigenvalue as $\phi$.  
 
 We refer the reader to \cite[Chap 2]{hida:book} or \cite{dembele-voight} for 
 details.\footnote{The standard references do not quite consider the
 orders we use, which were considered by Hijikata, Pizer and Shemanske.
 Consequently, one needs to check the behavior of conductors of local representations
 to guarantee that the Jacquet--Langlands correspondence
 associates to $\phi$ as above a Hilbert modular form $f$ whose level actually
 is $\frakN$. This is not hard, however, and the details may be found in \cite{me:basis}.}
 
%

\section{Congruence with Eisenstein series} \label{sec2}

%

We keep the notation of the previous section.  

For $f \in M_2(\frakN)$ (resp.\ $\phi \in \calM$) which is an 
eigenfunction of the Hecke operator $T_\ell$, let $\lambda_\ell(f)$ 
(resp.\ $\lambda_\ell(\phi)$) denote the corresponding eigenvalue.

Let $E$ be the normalized Eisenstein series on $M_2(\frakN)$
with Hecke eigenvalues $\lambda_{\ell}(E) = N(\ell) + 1$ when
$\ell \nmid \frakN$ and $\lambda_{\ell}(E) = 1$ when $\ell | \frakN$.  When $F=\Q$, this is 
the $E_{2,N}$ explicitly constructed in the introduction.  Recall $\phi_0 = 1 \in \calE$.
Then $\phi_0$ is an eigenform for the Hecke operators $T_\ell$ because the
Brandt matrices have constant row sums, and in fact
$\lambda_{\ell}(E) = \lambda_{\ell}(\phi_0)$  
when $\ell \nmid \frakN_2$ and $\ell^2 \nmid \frakN_1$.

The field of rationality of an eigenform $\phi \in \calM$ is the field $\Q(\phi)$ generated by 
its Hecke eigenvalues.  We may normalize $\phi$ so that all values of $\phi$ are integers
in $\Q(\phi)$, in which case we say $\phi$ is integral.
Let $\Q(\frakN)$ be the compositum of the fields of rationality $\Q(\phi)$ of all
eigenforms in $\phi \in \calM$.  Note this is a subfield of the compositum of all
$\Q(f)$ where $f$ ranges over eigenforms in $S_2(\frakN)$.  

If $\phi, \phi'$ are integral and $\frakp$ is an ideal of a suitable rationality field,
we write $\phi \equiv \phi' \mod \frakp$ if $\phi(\calI) \equiv \phi'(\calI) \mod \frakp$
for all $\calI \in \Cl(\calO)$.  If $\phi, \phi'$ are eigenforms, then $\phi(\calI) \equiv
\phi'(\calI) \mod \frakp$ implies their Hecke eigenvalues are also congruent mod $\frakp$ because the Hecke operators act by integral Brandt matrices.
We expect the converse to generally be true, though do not know how to show it.
See \cref{rem:converse} and \eqref{eq:hyp} below.

Denote by $h_F^+$ the narrow class number of $F$.

\begin{thm} \label{thm1} Assume $h_F^+ = h_F$ is odd.\footnote{The hypotheses
that $h_F^+$ is odd (which implies $h_F = h_F^+$) was added because of the corrected
description of $\calS$.}
Let $\num$ be the numerator of $m(\calO)$.  
Suppose $p | \num$ with $p \nmid h_F$\footnote{The condition $p \nmid h_F$
was added because of the corrected description of $\calM$.}
and $\frakp$ is a prime above $p$ in $\Q(\frakN)$.   Then there exists an eigenform $f \in S_{2}(\frakN)$ such that
$\lambda_\ell(f) \equiv \lambda_{\ell}(E) \mod \frakp$
 for all $\ell$ with $\ell \nmid \frakN_2$ and $\ell^2 \nmid \frakN_1$.  
 If $\frakN = \frakN_1$ is squarefree, we may take $f$ to be a newform.
 \end{thm}

\begin{proof} 

First we show there exists an integral $\phi \in \calS$ such that
$\phi \equiv 1 \mod p$.  Let us reindex our ideal representatives for $\Cl(\calO)$ in the
form $\calI_{ij}$, where for fixed $i$ the collection $\{\calI_{ij} \}_j$ is the set of ideal
representatives which differ from $\calI_{i1}$ by an element of $\Cl(\mathfrak o_F)$.
Consider $\phi \in M(\calO)$ integral with $\phi \equiv 1 \mod p$ so, for each 
$\calI_{ij} \in \Cl(\calO)$, $\phi(\calI_{ij}) = 1 + pa_{ij}$ for some $a_{ij} \in \Z$.   
Then, by \eqref{eq:M-def}, $\phi \in \calM$ just means: (i) $a_{ij} = a_{ik}$ for any valid
indices $ij$ and $ik$.  Moreover, since $h_F^+$ is odd, $\mathcal E = \C \phi_0$.  
Thus $\phi \in \calS$ if in addition $[\phi, \phi_0] = 0$, i.e., 
$\phi \in \calS$ if and only if (i) holds and we have
(ii) $m(\calO) = -p \sum_{ij} w_{ij}^{-1} a_{ij}$, where $w_{ij} =  [ \calO_\ell(\calI_{ij})^\times : 
\mathfrak o_F^\times]$. 

We claim there exist $a_i \in \Z$ such that these conditions hold.  
 (Note $p | \num$ implies $h > 1$ so necessarily $\calS \neq 0$.)
Let $w = \prod_{ij} w_{ij}$
and $w_{ij}^* = \frac{w}{w_{ij}}$.  We can rewrite (ii) as $\sum w_{ij}^* = -p \sum w_{ij}^* a_{ij}$.
To also account for (i), put $w_i' = \sum_j w_{ij}^*$ and  $b_i = a_{ij}$, and let $k$ be
the maximum value of the index $i$.
Then (i) and (ii) are achievable if and only if $\sum_{i=1}^k w_i' = -p \sum_{i=1}^k b_i w_i'$ 
is solvable  for $b_i \in \Z$.  
For this, it suffices to show $\gcd(w_1', \ldots, w_k')$ divides $p^{-1} \sum_{i=1}^k w_i'$, which is obvious at primes away from $p$.  
Say $p^j \parallel \gcd(w_1', \ldots, w_k')$.  
From the action of the $\Cl(\mathfrak o_F)$ on $\Cl(\calO)$,
we see that $w_{ij}^*$ is independent of the choice of $j$ and for each $i$ we have 
$w_i' = d_i w_{i1}^*$ for some $d_i | h_F$.  
Since we assumed $p \nmid h_F$, the $d_i$'s are prime to $p$, whence $p^j | w$.
Then $p | \num$ implies
$wp | \sum w_i' = w m(\calO)$, and therefore $p^{j+1} | \sum w_i'$.  This
proves the claim and gives us our desired $\phi$.

Now let $\Phi$ be the set of integral $\phi \in \calS$ which are congruent to a nonzero
multiple of $\phi_0$ mod $\frakp$. 
Fix a basis of eigenforms $\phi_1, \cdots, \phi_s$ of $\calS$.  Let $r$ be minimal such that, after a possible reordering of the $\phi_i$'s, there exists
$\phi \in \Phi$ with $\phi = c_1 \phi_1 + \cdots + c_r \phi_r$, $c_i \in \Q(\mathfrak N)$.  
Say $\phi \equiv c \phi_0
\mod \frakp$.  Then, for $\ell$ such that $\ell \nmid \frakN_2$ and $\ell^2 \nmid
\frakN_1$,
\[ [T_\ell - \lambda_\ell(\phi_j) ] \phi \equiv ( \lambda_\ell(\phi_0) - \lambda_\ell(\phi_j) ) c \phi_0 \mod
\frakp, \]
and thus $[T_\ell - \lambda_\ell(\phi_j) ] \phi \in \Phi$ unless $\lambda_\ell(\phi_0) \equiv \lambda_\ell(\phi_j)
\mod \frakp$.  (Note $[T_\ell - \lambda_\ell(\phi_j) ] \phi \in \Phi$ is also integral because
$T_\ell \phi$ is, so it makes sense to consider this mod $\frakp$.) 
However, $[T_\ell - \lambda_\ell(\phi_j) ] \phi$ is a linear combination of
$\phi_1, \ldots, \phi_{j-1}, \phi_{j+1}, \ldots, \phi_r$, which would contradict the minimality of $r$ if $[T_\ell - \lambda_\ell(\phi_j) ] \in \Phi$.  Hence 
$\lambda_\ell(\phi_0) \equiv \lambda_\ell(\phi_j) \mod \frakp$ for all $\ell$ as above 
and all $1 \le j \le r$.

Since $f$ has the same Hecke eigenvalues as $\phi$ for $T_\ell$ with $\ell$ as above,
 this yields the theorem.
\end{proof}

\begin{rem} \label{rem:thm1}

(a) The reason we assume $h_F^+ = h_F$ is to guarantee
the existence of $\phi \in \calS$ such that $\phi \equiv 1 \mod p$.  If
$h_F^+ \ne h_F$, one gets an additional linear constraint
$[\phi, \psi \circ N]=0$ on the $a_i$'s in the proof for each nontrivial character $\psi$ of 
$\Cl^+(\frako_F)/\Cl(\frako_F)$.
Then, at least a priori, one needs to place some conditions on the $w_i$'s, $m(\calO)$
and $p$ to guarantee the existence of $a_i \in \Z$ solving this system of 
$\Q$-linear equations.

(b) When $h_F^+ > 1$, there are other Eisenstein series in $M(\calO)$, and one 
might ask for congruences of these Eisenstein series as well.  However, such Eisenstein 
series which are also eigenforms can be obtained by 
twisting $\phi_0$ by narrow ideal class characters of $F$, 
and analogous congruences are just obtained by twisting both $E$ and $f$.
\end{rem}

We observe that there are often multiple choices for $B$ for a given $\frakN$.
For instance, if $F=\Q$ and $N=11 \cdot 13$, we get $\num = 5 \cdot 7$ if we take 
$N_1 = 11$, $N_2 = 13$ and $\num = 12$ if we take $N_1 = 13$, $N_2 = 11$.
Thus one gets congruences of cusp forms in level $11 \cdot 13$ 
with Eisenstein series modulo the primes $p = 2, 3, 5, 7$.

Some examples with $F=\Q$ were given in the introduction.
Here is a simple example with $F\neq \Q$.

\begin{ex} Let $F = \Q(\sqrt{5})$, which has narrow class number $1$.  
Suppose $\frakN = \frakN_2$ is one of the two prime ideals
with $N(\frakN) = 31$.  Then $\zeta_F(-1) = \frac 1{30}$ and $m(\calO) = \frac{32}{60} = \frac 8{15}$, 
so we take $p=2$.  Here
$S_2(\frakN)$ is one dimensional and has rationality field $\Q$.  
Then for nonzero $f \in S_2(\frakN)$ our theorem says that $\lambda_\ell(f) \equiv N(\ell) + 1 \mod 2$ 
for $\ell \ne \frakN$.  It is also true that $\lambda_\frakN (f) \equiv 1 \mod 2$.
Indeed, the first few nonzero Hecke 
eigenvalues  $\lambda_\ell(f)$ satisfy the values listed in the following table:

\begin{center}
\begin{tabular}{c|ccccccccccccc}
$N(\ell)$ & $4$ & $5$ & $9$ & $11$ & $11$ & $19$ & $19$ & $29$ & $29$ & $31$ & $31$ & $41$ & $41$ \\
\hline
$\lambda_\ell(f)$ & $-3$ & $-2$ & $2$ & $4$ & $-4$ & $-4$ & $4$ & $-2$ & $-2$ & $-1$ & $8$ & $-6$ & $-6$
\end{tabular}
\end{center}
\end{ex}

\begin{rem} \label{rem:converse}
Suppose $\frakN= \frakN_1$ is squarefree.  If $f \in S_2(\frakN)$
such that $f \equiv E \mod \frakp$, we expect (cf.\ \eqref{eq:hyp} below) that $f$ 
corresponds to a quaternionic
form $\phi \in \calS$ such that, after normalization, $\phi \equiv 1 \mod \frakp$.
Let $p$ be the rational prime below $\frakp$.
For such a cusp form $\phi$ to exist, we need $w[\phi, \phi_0] \equiv w[\phi_0, \phi_0]
\equiv 0 \mod \frakp$ where as before $w = \prod w_i$.  Thus if $(p, w) = 1$,
a congruence $f \equiv E \mod \frakp$ should only exist if $p | \num$ as in the theorem.  If $F=\Q$, then no prime $ > 3$ divides $w$, and this agrees with the
converse obtained by Ribet.

We also note that when $p | \num$, if there is a unique newform $f \in S_2(\frakN)$
such that $f \equiv E \mod \frakp$, the above proof implies $f$ corresponds to
a quaternionic form $\phi \in \calS$ with $\phi \equiv 1 \mod \frakp$, as we must
have $r=1$ in the final paragraph.  When $r=1$, $F=\Q$ and $N=\frakN$ is prime,
this is also true by \cite{quattrini}, \cite{emerton}.
\end{rem}

%

\section{Quadratic twist $L$-values} \label{sec3}

%

We keep our previous notation, but now assume $\calO$ is a maximal order of $B$, i.e.,
$\frakN = \frakN_1$ and $\frakN$ is squarefree.  For consistency and clarity, 
we will denote complete $L$-functions by $L^*(s, -)$ and incomplete (i.e., the finite part of) 
$L$-functions just by $L(s, -)$.  However, we follow the usual convention that automorphic 
$L$-functions are normalized
so that $s=1/2$ is the central point, whereas $L$-functions for $f \in S_2(\frakN)$
are normalized classically with $s=1$ the central point.

Let $K/F$ be a CM quadratic extension such that each $\frakp | \frakN$ is inert in $K$.  
This means $K$ embeds in $B$ and, since we
are now assuming $\calO$ is maximal, we may fix an embedding of $K$ into $B$
such that $\frako_K \subset \calO$. Hence, to each 
$t \in \Cl(\frako_K) = K^\times \bs \hat K^\times / \hat \frako_K^\times$ we can associate the
quaternionic ideal class $x(t) \in \Cl(\calO)$ via $x(t) = B^\times t \hat \calO^\times$.
Fix a character $\chi$ of $\Cl(\frako_K)$.  Put the product of local Tamagawa measures
on $\A_F^\times$ and $\A_K^\times$, and the counting measure on $K^\times$.
For $\phi \in \calM$, define the period
\[ P_\chi(\phi) = \int_{K^\times \A_F^\times \bs \A_K^\times} \phi(t) \overline{\chi(t)} \, dt = 
\frac{\vol(K^\times \A_F^\times \bs \A_K^\times)}{h_K}  \sum_{t \in \Cl(\frako_K)} \phi(x(t)) \overline{\chi(t)}, \]
where $dt$ is the quotient measure.
Let $\eta_K$ denote the quadratic idele class character of $F$ associated to $K/F$.
Then
\[  \frac{ \vol(K^\times \A_F^\times \bs \A_K^\times)}{h_K}
= \frac{2L^*(1,\eta_K)}{h_K} = \frac{2^{d+1} \sqrt{\Delta_F} }{w_K  h_F Q_{K/F} \sqrt{|\Delta_K|}}, \]
where $Q_{K/F} = [\frako_K^\times: \frako_F^\times \mu(K)] = 2^{d-1} \frac{R_F}{R_K}$,
the $R_*$ denoting a regulator.
We remark $Q_{K/F}$ is $1$ or $2$.
Consider the normalized period,
\begin{equation}
 P^0_\chi(\phi) = \sum_{t \in \Cl(\frako_K)} \phi(x(t)) \overline{\chi(t)} 
= \frac{w_K  h_F Q_{K/F} \sqrt{|\Delta_K|}}{2^{d+1}\sqrt{\Delta_F} } P_\chi(\phi).
\end{equation}

Let $\phi \in \calS$ be an eigenform.  
Then $\phi$ generates an irreducible cuspidal automorphic representation $\pi$ of 
$B^\times(\A_F)$ with trivial central character.  The period
$P_\chi$ extends to a linear functional on $\pi$ via the same defining integral.  
Suppose $P_\chi(\phi) \ne 0$ so $P_\chi$ is a nonzero functional. 
Then for each place $v$ of $F$, $\pi_v$ has a nonzero local $K_v^\times$-invariant linear
functional $\ell_v$.  It is well known that the functional $\ell_v$ is unique up to scaling,
and we normalize the $\ell_v$'s so that $P_\chi = \prod \ell_v$ (i.e., this 
factorization holds factorizable forms in $\pi$).

For each finite place $v$, Gross and Prasad
\cite{gross-prasad} defined a test vector $\phi_v \in \pi_v$ such that $\ell_v(\phi_v) \ne 0$.
In our setting, the Gross--Prasad test vector $\phi_v$ is the (unique up to scaling)
vector in $\pi_v$ fixed by $\calO_v^\times$.  (This is not true if $\calO$ is not maximal.)
This implies that $\phi = \prod \phi_v$ as a function on $B^\times(\A_F)$ where
$\phi_v$ is a suitably normalized (to guarantee convergence) Gross--Prasad test
vector for $v < \infty$ and $\phi_v = 1$ for $v | \infty$.

The central-value formula in \cite{mw}, refining Waldspurger's formula \cite{wald}, specializes here to
\begin{equation} \label{eq:mw}
\frac{|P_\chi(\phi)|^2}{(\phi, \phi)} = \frac {\zeta_F(2)}{2\pi^{2d}} \sqrt{\frac{\Delta_F}{|\Delta_K|}} 
\prod_{v | \frakN} (1-q_v^{-1})  
 \frac{L^*(1/2, \pi_K \otimes \chi)}{L^*(1,\pi, Ad)}.
\end{equation}
(Note \cite{mw} uses a slightly different choice of measure 
on $B^\times(\A)$ there---a formulation with the present choice of
measure is given in \cite{fmp}.)  
Observe the left hand side of this formula is invariant under scaling.
The $L$-functions for $\pi$ are the same as the $L$-functions for its 
Jacquet--Langlands transfer to $\GL(2)$.

We will use this formula to deduce a precise congruence result on $L$-values.
We remark that Zhang (e.g., \cite{zhang}), generalizing the work of Gross \cite{gross}, 
also obtained an explicit $L$-value formula in this 
setting.  It is formulated somewhat differently than \eqref{eq:mw}, but presumably 
Zhang's formula can be used in a similar manner.  Alternatively one could use
half-integral weight forms and the formula of Baruch--Mao \cite{BM} (generalizing
an earlier formula of Waldspurger), similar to the approach in \cite{quattrini} for $F=\Q$.

Unfortunately, we are not able to prove that a Hecke eigenvalue congruence implies
a congruence of $L$-values as \cite{vatsal}, \cite{hv} do over $\Q$.  
Our method requires starting with a congruence of quaternionic modular forms.  In
fact we generally expect the statement
\begin{equation} \label{eq:hyp}
\lambda_\ell(\phi) \equiv \lambda_\ell(\phi_0) \mod \frakp \text{ for all } \ell \implies
c \phi \equiv \phi_0 \mod \frakp \text{ for suitable } c
\end{equation}
to hold, but we do not know how to prove it.  As remark earlier, the converse direction
is true.  Recall from \cref{rem:converse} that at least \eqref{eq:hyp}
is true if there is a unique (up to scaling) eigenform $\phi \in \calS$
satisfying the above Hecke eigenvalue congruence, i.e.\
a unique cuspidal newform $f$ congruent to $E$).  This is also true if $F=\Q$
and $N=\frakN$ is prime (\cite[Thm 3.6]{quattrini}, relying on \cite{emerton}).

Our main evidence for believing \eqref{eq:hyp} holds more generally is that 
it seems necessary to obtain a congruence of $L$-values
 for newforms $f$ congruent to $E$ as in \cite{hv}.  This is because a congruence of 
 $L$-values would say that (up to constants) we have
 $|P^0_K(\phi)|^2 \equiv |P^0_K(\phi_0)|^2 \mod \frakp$ for infinitely many $K$.
 On the other hand, the map $\Cl(\frako_K) \to \Cl(\calO)$ appears to behave 
 essentially randomly, so an infinitude of such congruences would seem to happen
 with probability 0 if $c\phi \not \equiv \phi_0$.
 We remark that  the expectation
 \eqref{eq:hyp} relies on the specific action of the Hecke algebra on $\calM$---e.g., if 
 one were to replace the action of the Brandt matrices on $\calM$ by diagonal 
 matrices 
 the analogous statement fails---so  \eqref{eq:hyp}  may be
 a rather deep arithmetic statement  
 (indeed, the analysis in \cite{emerton} is not trivial).

\medskip
Before deducing our precise congruence result of $L$-values under a congruence of
quaternionic modular forms we observe that we can at least get some statement about
nonvanishing of $L$-values among forms congruent to Eisenstein series.  For this we can also
allow $K/F$ to be ramified at primes dividing the level.

\begin{prop} \label{prop:nonv}
We continue the hypotheses and notation of \cref{thm1}, 
with the further assumption that $\frakN = \frakN_1$ is squarefree.  Let
$K/F$ be a CM extension not split at each prime dividing $\frakN$.  If $p \nmid h_K$, then there exists a newform
$f \in S_2(\frakN)$ such that $\lambda_\ell(f) \equiv \lambda_\ell(E) \mod \frakp$
for all $\ell$ and $L(1,f_K) \ne 0$.
\end{prop}

When $F=\Q$, $p > 2$, and $p \nmid N = \frakN$, 
one can deduce stronger nonvanishing results from \cref{thm1} using \cite{hv}.

\begin{proof} By the proof of \cref{thm1}, there exists an integral
$\phi \in \calS$ such that $\phi \equiv 1 \mod \frakp$ and $\phi = \sum \phi_i$ where
the $\phi_i$ are eigenforms with $T_\ell$-eigenvalues congruent to $\lambda_\ell(E) \mod \frakp$.  Now if $p \nmid h_K$, then
\[ P^0(\phi) \equiv \sum_{t \in \Cl(\frako_K)} 1 \equiv h_K \mod \frakp \]
implies $P^0(\phi) \ne 0$.  Since $P^0(\phi) = \sum P^0(\phi_i)$ at least one $P^0(\phi_i)$
is nonzero.  Hence, by Waldspurger's formula \cite{wald} (or the more refined \eqref{eq:mw}),
$L(1,f_{i,K}) \ne 0$ where $f_i \in S_2(\frakN)$ is the newform
corresponding to $\phi_i$ via Jacquet--Langlands.  (While we have only stated \eqref{eq:mw}
if each prime dividing $\frakN$ is inert in $K$, \cite{wald} and \cite{mw} allow these
primes to be ramified in $K$, with a suitable change in local factors.)
\end{proof}

From now on we will suppose that we have a newform 
$f \in S_2(\frakN)$ corresponds to
an integral eigenform $\phi \in \calS$ such that $\phi \equiv 1 \mod \frakp^r$
for some $r \ge 1$.
As remarked above, this holds for $r=1$ under the conditions of \cref{thm1}
if there is a unique cuspidal newform $f$ congruent to $E$ mod $\frakp$ or if
$F=\Q$ and $N=\frakN$ is prime.

Consider the Petersson norm on $S_2(\frakN)$ normalized so that
\[ (f,f) = 2^{1-2d} h_F \Delta_F^2 N(\frakN) L^*(1, \pi, Ad), \]
which corresponds to the definition of $(f,f)$ in \cite{hida:91} (cf.\ \cite[Thm 5.16]{GG}).  Here
 $\pi$ is the automorphic representation associated to $f$.  Then \eqref{eq:mw} becomes
\begin{equation}
\frac{|P_\chi(\phi)|^2}{(\phi, \phi)}
=  \frac{\zeta_F(2)\Delta_F^{5/2} h_F N(\frakN)}{(2\pi^2)^{2d} \sqrt{|\Delta_K|} }
\prod_{v | \frakN} (1-q_v^{-1})  
 \frac{L(1, f, \chi)}{(f,f)}.
\end{equation}

 
We normalize $\phi$ so that the gcd of all values of $\phi$ is minimal.
This specifies $\phi$ only up to a root of unity in $\Q(\phi)$, but
specifies both $|P_\chi(\phi)|^2$ and $(\phi, \phi)$ uniquely.  Now it need not
be that $\phi \equiv 1 \mod \frakp$, but we will have $\phi \equiv c_\phi \mod \frakp$
for some integer $c_\phi$ of $\Q(\phi)$ which is nonzero mod $\frakp$.
With this normalization, put
 \begin{equation} \label{eq:Lalg-def}
  L^\alg(1,f, \chi) = |P_\chi^0(\phi)|^2 = \frac{h_F^2 w_K^2 Q_{K/F}^2 \sqrt{|\Delta_K|}}
  {4(2\pi)^{2d}} [\phi, \phi] \frac{L(1,f, \chi)}{(f,f)},
 \end{equation}
  which is an algebraic integer since $P^0_\chi(\phi)$ is.  

Note that our algebraic special value $L^\alg(1,f, \chi)$ is normalized differently 
from what is typically found in the literature. 
For instance, when $F=\Q$, Shimura 
\cite{shimura} essentially considered the values
\[ A(1, f, \chi) = -\frac{g(\eta_K)}{(2\pi)^2 } \frac{L(1,f, \chi)}{ \langle f, f \rangle }, \]
where $g({ \cdot })$ denotes the Gauss sum and $\langle f, f \rangle = \frac 12 (f, f)$
is the usual Petersson norm.  In our case, with $F=\Q$, we see
\begin{equation} L^\alg(1,f, \chi) = - \frac{w_K^2}{8}   \frac  {\sqrt{|\Delta_K|}}{g(\eta_K)} [\phi, \phi] A(1, f, \chi),
\end{equation}
which, for fixed $f$, just depends on $K$ in a simple way (and not at all on $\chi$).
Note the factor $[\phi, \phi] = \sum w_i^{-1} |\phi(\calI_i)|^2$ is also algebraic.

\begin{thm} \label{thm2} Suppose $\frakN = \frakN_1$ is squarefree and $f \in S_2(\frakN)$
is a newform corresponding to a quaternionic form $\phi \in \calS$ such that $\phi \equiv 1 \mod \frakp^r$.
  Then there exists an algebraic integer $c_\phi \in \Q(\phi)$ which is nonzero
  mod $\frakp$ such that, for any quadratic
extension $K/F$ which is inert at each prime dividing $\frakN$ and
any ideal class character $\chi$ of $K$,
\[ L^{\alg}(1, f, \chi) \equiv  \delta_{\chi, 1} |c_\phi|^2 h_K^2 \mod \frakp^r, \]
where $\delta_{\chi,1}$ is $1$ if $\chi = 1$ and $0$ otherwise.
In particular, if $\frakp^r \nmid h_K^2$, then $L(1, f_K) \ne 0$.
\end{thm}

We do not require $\frakp | \num$ as in \cref{thm1}.
One could also allow $K$ to be ramified at primes dividing $\frakN_1$, but
then one needs to exclude the $L$-factors at such primes (cf.\ \cite{mw}).

Moreover, the number $c_\phi$ can be read off directly from
$\phi$ (and can be chosen independent of $\frakp$), which can be determined from Brandt matrices.  The calculations of $c_\phi$ and $[\phi, \phi]$ (and thus the
normalization of $L^\alg(1,f,\chi)$)
may be simpler than the calculation of canonical periods arising in, e.g., 
\cite{hv}.

\begin{proof}
With $\phi$ normalized as above, we get
$P^0_\chi(\phi) \equiv c_\phi P^0_\chi(\phi_0) \mod \frakp^r$, and $P^0_\chi(\phi_0)$ is either $h_K$ or $0$ according to whether $\chi$ is 
trivial or not.  Now apply \eqref{eq:Lalg-def}.
\end{proof}

Finally, we briefly illustrate how one can use $L$-values to recover 
information about the map $\Cl(\frako_K) \to \Cl(\calO)$, and use this to give
an example where $L(1/2, f_K)^\alg \not \equiv 0 \mod p$ and $L(1/2, f, \chi) \ne 0$
but $L(1/2, f, \chi) \equiv 0 \mod p$ for a nontrivial ideal class character $\chi$ of $K$.

\begin{ex} Suppose $F=\Q$, $N=11$ and $p=5$.  Then 
the quaternion algebra $B$ ramified
at $11$ and $\infty$ has class number $2$.    Write $\Cl(\calO) = \{ x_1, x_2 \}$.  Then, up to reordering, $w_1 = 3$ and
$w_2 = 2$.  Here $\dim \calS = 1$, and we can define $\phi \in \calS$
by  $\phi(x_1) = 3$ and $\phi(x_2) = -2$.  This is normalized as we specified above,
and $[\phi, \phi] = 5$ and we can take $c_\phi = 3$.

Let $K=\Q(\sqrt{-23})$, which has class number 3,
and $\calO$ be a maximal order of $B$ containing 
$\frako_K$.
Our theorem says that 
$L^\alg(1, f_K) \equiv -h_K^2 \equiv 1 \mod 5$, where $f \in S_2(11)$ is the unique
normalized cusp form.  Indeed, one can compute $L(1/2, f_K)^\alg = 1$.
On the other hand, the normalized period for $\chi=1$ is 
$P^0_1(\phi) = \sum_{t \in \Cl(\frako_K)} \phi(x(t)) = 5a - 6$, 
where $0 \le a \le 3$ is the number of classes of $\frako_K$ mapping to $x_1$.
Since $|P^0_1(\phi)|^2 = 1$, we deduce $a=1$.   
Hence $x(t_1) = x_1$ and $x(t_2) = x(t_3) = x_2$ for some ordering $t_1, t_2, t_3$
of the ideal classes of $\frako_K$, i.e., the map $\Cl(\frako_K) \to \Cl(\calO)$ is surjective
with fibers of size $1$ and $2$ over $x_1$ and $x_2$, respectively.

Now let $\chi$ be one of the nontrivial idele class characters.  Then the
$\zeta_i :=\chi(t_i)$'s are the distinct 3rd roots of unity in some order, and 
$P^0_\chi(\phi) = 3 \zeta_1 -2 \zeta_2 - 2 \zeta_3 = 5 \zeta_1$.  Hence we see
$L^\alg(1, f, \chi)$ is nonzero, but $0 \text{ \rm{mod} } 5$.
\end{ex}

We remark that with $f$ as in the previous example, 
one can use the values of $\phi$ to conclude
$L^\alg(1, f_K) \in \{ (5a-2h_K)^2 : 0 \le a \le h_K \}$ for any imaginary quadratic $K$ 
unramified  at 11.

\begin{rem}
One can similarly use this method to prove that congruences of quaternionic
cusp forms yield congruences of $L$-values.  That is, if $\phi_1$ and $\phi_2$ are integral
quaternionic eigenforms such that $\phi_1 \equiv \phi_2 \mod \frakp^r$, then with a
suitable definition of algebraic $L$-values, one will have $L^\alg(1,f_1,\chi) \equiv
L^\alg(1, f_2, \chi) \mod \frakp^r$.  See \cite{DK} for an approach using half-integral
weight forms to this situation when $F=\Q$.
\end{rem}


\begin{bibdiv}
\begin{biblist}*{labels={alphabetic}}

\bib{BM}{article}{
   author={Baruch, Ehud Moshe},
   author={Mao, Zhengyu},
   title={Central value of automorphic $L$-functions},
   journal={Geom. Funct. Anal.},
   volume={17},
   date={2007},
   number={2},
   pages={333--384},
   issn={1016-443X},
}

\bib{bkk}{article}{
   author={Berger, Tobias},
   author={Klosin, Krzysztof},
   author={Kramer, Kenneth},
   title={On higher congruences between automorphic forms},
   journal={Math. Res. Lett.},
   volume={21},
   date={2014},
   number={1},
   pages={71--82},
   issn={1073-2780},
}

\bib{dembele-voight}{article}{
   author={Demb{\'e}l{\'e}, Lassina},
   author={Voight, John},
   title={Explicit methods for Hilbert modular forms},
   conference={
      title={Elliptic curves, Hilbert modular forms and Galois deformations},
   },
   book={
      series={Adv. Courses Math. CRM Barcelona},
      publisher={Birkh\"auser/Springer, Basel},
   },
   date={2013},
   pages={135--198},
}


\bib{DK}{article}{
   author={Dummigan, Neil},
   author={Krishnamoorthy, Srilakshmi},
   title={Lifting congruences to weight 3/2},
   journal={J. Ramanujan Math. Soc.},
   volume={32},
   date={2017},
   number={4},
   pages={431--440},
   issn={0970-1249},
}

\bib{emerton}{article}{
   author={Emerton, Matthew},
   title={Supersingular elliptic curves, theta series and weight two modular
   forms},
   journal={J. Amer. Math. Soc.},
   volume={15},
   date={2002},
   number={3},
   pages={671--714 (electronic)},
   issn={0894-0347},
}


\bib{fmp}{article}{
   author={File, Daniel},
   author={Martin, Kimball},
   author={Pitale, Ameya},
   title={Test vectors and central $L$-values for ${\rm GL}(2)$},
   journal={Algebra Number Theory},
   volume={11},
   date={2017},
   number={2},
   pages={253--318},
   issn={1937-0652},
}

\bib{GG}{book}{
   author={Getz, Jayce},
   author={Goresky, Mark},
   title={Hilbert modular forms with coefficients in intersection homology and
quadratic base change},
   series={Progress in Mathematics},
   volume={298},
   publisher={Birkh\"auser/Springer Basel AG, Basel},
   date={2012},
   pages={xiv+256},
   isbn={978-3-0348-0350-2},
}

\bib{ghate}{article}{
   author={Ghate, Eknath},
   title={Adjoint $L$-values and primes of congruence for Hilbert modular
   forms},
   journal={Compositio Math.},
   volume={132},
   date={2002},
   number={3},
   pages={243--281},
   issn={0010-437X},
}

\bib{gross}{article}{
   author={Gross, Benedict H.},
   title={Heights and the special values of $L$-series},
   conference={
      title={Number theory},
      address={Montreal, Que.},
      date={1985},
   },
   book={
      series={CMS Conf. Proc.},
      volume={7},
      publisher={Amer. Math. Soc., Providence, RI},
   },
   date={1987},
   pages={115--187},
}

\bib{gross-prasad}{article}{
   author={Gross, Benedict H.},
   author={Prasad, Dipendra},
   title={Test vectors for linear forms},
   journal={Math. Ann.},
   volume={291},
   date={1991},
   number={2},
   pages={343--355},
   issn={0025-5831},
}

\bib{hv}{article}{
   author={Heumann, Jay},
   author={Vatsal, Vinayak},
   title={Modular symbols, Eisenstein series, and congruences},
   journal={J. Th\'eor. Nombres Bordeaux},
   volume={26},
   date={2014},
   number={3},
   pages={709--757},
   issn={1246-7405},
}

\bib{hida}{article}{
   author={Hida, Haruzo},
   title={Congruence of cusp forms and special values of their zeta
   functions},
   journal={Invent. Math.},
   volume={63},
   date={1981},
   number={2},
   pages={225--261},
   issn={0020-9910},
}

\bib{hida:91}{article}{
   author={Hida, Haruzo},
   title={On $p$-adic $L$-functions of ${\rm GL}(2)\times {\rm GL}(2)$ over
totally real fields},
   language={English, with French summary},
   journal={Ann. Inst. Fourier (Grenoble)},
   volume={41},
   date={1991},
   number={2},
   pages={311--391},
   issn={0373-0956},
}

\bib{hida:book}{book}{
   author={Hida, Haruzo},
   title={Hilbert modular forms and Iwasawa theory},
   series={Oxford Mathematical Monographs},
   publisher={The Clarendon Press, Oxford University Press, Oxford},
   date={2006},
   pages={xiv+402},
   isbn={978-0-19-857102-5},
   isbn={0-19-857102-X},
}

\bib{me:basis}{unpublished}{
      author={Martin, Kimball},
   title={The basis problem revisited},
   note={arXiv:1804.04234},
}

\bib{mw}{article}{
   author={Martin, Kimball},
   author={Whitehouse, David},
   title={Central $L$-values and toric periods for ${\rm GL}(2)$},
   journal={Int. Math. Res. Not. IMRN},
   date={2009},
   number={1},
   pages={141--191},
   issn={1073-7928},
}

\bib{mazur1}{article}{
   author={Mazur, B.},
   title={Modular curves and the Eisenstein ideal},
   journal={Inst. Hautes \'Etudes Sci. Publ. Math.},
   number={47},
   date={1977},
   pages={33--186 (1978)},
   issn={0073-8301},
}

\bib{mazur2}{article}{
   author={Mazur, B.},
   title={On the arithmetic of special values of $L$ functions},
   journal={Invent. Math.},
   volume={55},
   date={1979},
   number={3},
   pages={207--240},
   issn={0020-9910},
}

\bib{michel-ramakrishnan}{article}{
   author={Michel, Philippe},
   author={Ramakrishnan, Dinakar},
   title={Consequences of the Gross-Zagier formulae: stability of average $L$-values, subconvexity, and non-vanishing mod $p$},
   conference={
      title={Number theory, analysis and geometry},
   },
   book={
      publisher={Springer, New York},
   },
   date={2012},
   pages={437--459},
}

\bib{quattrini}{article}{
   author={Quattrini, Patricia L.},
   title={The effect of torsion on the distribution of Sh among quadratic
   twists of an elliptic curve},
   journal={J. Number Theory},
   volume={131},
   date={2011},
   number={2},
   pages={195--211},
   issn={0022-314X},
}

\bib{shimura}{article}{
   author={Shimura, Goro},
   title={The special values of the zeta functions associated with cusp
   forms},
   journal={Comm. Pure Appl. Math.},
   volume={29},
   date={1976},
   number={6},
   pages={783--804},
   issn={0010-3640},
}

\bib{stevens}{book}{
   author={Stevens, Glenn},
   title={Arithmetic on modular curves},
   series={Progress in Mathematics},
   volume={20},
   publisher={Birkh\"auser Boston, Inc., Boston, MA},
   date={1982},
   pages={xvii+214},
   isbn={3-7643-3088-0},
}

\bib{vatsal}{article}{
   author={Vatsal, V.},
   title={Canonical periods and congruence formulae},
   journal={Duke Math. J.},
   volume={98},
   date={1999},
   number={2},
   pages={397--419},
   issn={0012-7094},
}

\bib{vigneras}{book}{
   author={Vign{\'e}ras, Marie-France},
   title={Arithm\'etique des alg\`ebres de quaternions},
   series={Lecture Notes in Mathematics},
   volume={800},
   publisher={Springer, Berlin},
   date={1980},
   pages={vii+169},
   isbn={3-540-09983-2},
}

\bib{wald}{article}{
   author={Waldspurger, J.-L.},
   title={Sur les valeurs de certaines fonctions $L$ automorphes en leur
   centre de sym\'etrie},
   journal={Compositio Math.},
   volume={54},
   date={1985},
   number={2},
   pages={173--242},
   issn={0010-437X},
}

\bib{yoo:1}{unpublished}{
   author={Yoo, Hwajong},
   title={Non-optimal levels of a reducible mod $\ell$ modular representation},
   note={arXiv:1409.8342v3},
}

\bib{yoo:2}{unpublished}{
   author={Yoo, Hwajong},
   title={Rational {E}isenstein primes and the rational cuspidal groups of modular {J}acobian varieties},
   note={arXiv:1510.0301v3},
}


\bib{zhang}{article}{
   author={Zhang, Shou-Wu},
   title={Gross-Zagier formula for $\rm GL(2)$. II},
   conference={
      title={Heegner points and Rankin $L$-series},
   },
   book={
      series={Math. Sci. Res. Inst. Publ.},
      volume={49},
      publisher={Cambridge Univ. Press, Cambridge},
   },
   date={2004},
   pages={191--214},
}

\end{biblist}
\end{bibdiv}
\end{document}